\documentclass[a4paper,12pt]{amsart}
\usepackage{amsmath,amssymb,amscd,amsfonts}
\usepackage{amsrefs, esint}
\numberwithin{equation}{section}
\newtheorem{theorem}{Theorem}[section]

\newtheorem{rem}[theorem]{Remark}

\newtheorem{lemma}[theorem]{Lemma}

\newtheorem{prop}[theorem]{Proposition}
\newtheorem{conjecture}[theorem]{Conjecture}
\newtheorem{corollary}[theorem]{Corollary}

\def\det{\mathop{\rm det}\nolimits}

\def\dbar{\bar\partial}
\def\ddbar{\partial\bar\partial}
\def\d{\partial}

\def\cG{{\mathcal G}}

\def\cE{{\mathcal E}}

\def\cM{{\mathcal M}}

\def\cD{{\mathcal D}}

\def\cZ{{\mathcal Z}}

\def\cH{{\mathcal H}}
\let\ol=\overline

\let\ep=\varepsilon
\let\vp=\varphi 

\def\bC{{\mathbb C}}
\def\bR{{\mathbb R}}

\def\a{{\alpha}}
\def\z{{\zeta}}

\def\hg{\tilde{\gamma}}

\def\g{\gamma}

\def\b{\beta}

\def\k{{\kappa}}

\title[Strict Convexity]
{Strict convexity of the Mabuchi functional for energy minimizers }

\author{Long Li}
\address{Institute Fourier, 100 rue des maths 38610 Gi\`eres, Grenoble, France}
\email{Long.Li1@univ-grenoble-alpes.fr}

\begin{document}
\maketitle 

\begin{abstract}
There are two parts of this paper.
First, we discovered an explicit formula for the complex Hessian of the weighted log-Bergman kernel on a parallelogram domain,
and utilised this formula to give a new proof about the strict convexity of the Mabuchi functional along a smooth geodesic.
Second, when a $C^{1,1}$-geodesic connects two non-degenerate energy minimizers, 
we also proved this strict convexity, by showing that such a geodesic must be non-degenerate and smooth. 
\end{abstract}

\section{Introduction}

Suppose $X$ is an $n$-dimensional compact complex K\"ahler manifold,
and $\omega$ is its associated K\"ahler form. 
Let $\cH$ be the space of all smooth K\"ahler potentials of $\omega$, i.e. 
$$\cH: =\{ \vp\in C^{\infty}(X); \ \ \ \omega+ i\ddbar\varphi > 0  \}. $$
Up to a chosen normalization, it can be identified with the space of all K\"aher metrics in the cohomology class $[\omega]$.
Moreover, the space $\cH$ becomes an infinite dimensional Riemannian manifold, after equipped with an $L^2$ metric on its tangent space.

A sub-geodesic $\cG$ in the space $\cH$ is 
an $S^1$-invariant positive $(1,1)$ current 
on $X\times [0,1] \times S^1  $ with the following form
$$ \cG: = \pi^*\omega+ dd^c\Phi \geq 0, $$
and $\cG$ is a geodesic connecting two points $\vp_0, \vp_1\in \cH$,
if it is a sub-geodesic and satisfies the following equation on $X\times [0,1]\times S^1$
with the boundary value $\Phi|_{X\times \{0\} } = \vp_0, \Phi|_{X\times\{1\} } = \vp_1$
\begin{equation}
\label{intr-010}
 \cG^{n+1} = 0.
\end{equation}
According to Chen \cite{C00}, such a geodesic always exists with an arbitrary boundary value in $\cH$,
but the current $\cG$ can be degenerate and has merely $C^{1,1}$ regularities in general. 

Sometimes we view $\cG$ as a curve of K\"ahler potentials defined on $X\times [0,1]$, and call it as a \emph{geodesic segment}.
Similarly, if a geodesic $\cG$ is defined on $X\times [0, +\infty)\times S^1$, then we call it as a \emph{geodesic ray}. 

On a Fano manifold,
Berndtsson \cite{Bo11} proved that the so called \emph{Ding functional} $\cD$ is convex
along any geodesic $\cG$. 
Amazingly, this proof is so strong that it also implies the following two facts:
first, $\cD$ is also convex along any sub-geodesic;  
second, $\cD$ is  \emph{strictly convex} along any geodesic (segment or ray),
in the sense that $\cG$ is generated by a holomorphic vector field whenever $\cD$ is linear along it.
These convexity results turned out to be very useful in the study of uniqueness and existence problems of the \emph{K\"ahler-Einstein} metrics.

Recently in the work of Berman-Berndtsson \cite{BB} and Chen-Li-P\u aun \cite{CLP},
the Mabuchi functional $\cM$ is proved to be convex and continuous along a $C^{1,1}$-geodesic segment $\cG$. 
The key observation in \cite{BB} is that the complex Hessian of the Mabuchi functional 
can be approximated weakly by the push forward of the following closed positive $(n+1, n+1)$ currents as $m\rightarrow \infty$
$$ i\ddbar \log K_t^{(m)}\wedge \cG^n. $$
Here $K_t^{(m)}(z,z)$ is the Bergman kernel of a unit ball in $\bC^n$ with the weight $m\phi$, where $\cG: = i\ddbar \phi$ locally.

The above $(n+1, n+1)$ currents are positive, since Berndtsson \cite{Bo} proved that 
\begin{equation}
\label{intro-0001}
i\ddbar_{t,z}\log K_t\geq 0. 
\end{equation}
for any plurisubharmonic(\emph{psh})  weight $\phi$ on a pseudoconvex domain in $\bC_{t}\times \bC^n_z$.
However, the computation in \cite{Bo} stoped at calculating the $(1,1)$ current 
$i\ddbar_t K_t(z,z) $
on a product domain $D:= U\times V$,
and then Berndtsson proved the positivity 
by some other tricks which avoid using the explicit formula of the complex Hessian in equation (\ref{intro-0001}).

The first part of this paper is aimed to give such an explicit formula in a parallelogram  domain of $\bC_{t}\times \bC^n_z$.
Let $v:=[1,a]$ be a vector in $\bC_t \times \bC^n_z$, and a parallelogram domain $D$ with slope $v$ is defined fiberwisely to be 
$$ D_t: = \{ z+ ta; \ \ z\in D_0\},$$
where $D_0$ is a strictly pseudoconvex domain in $\bC^n$. 
Then we found the following formula 
\begin{theorem}
\label{thm-intr-001}
Let $\phi$ be a smooth strictly plurisubharmonic weight function on the parallelogram  domain $D$.
Then the log-Bergman kernel of $D$ has the following complex Hessian lower bound at a point $z$ acting in the direction $v$
\begin{equation}
\label{intr-001}
i\ddbar \log K_t (v,\bar v) \geq \frac{ [i\ddbar\phi_a \wedge T_a]_{D_{a,t}}}{ K_{t}(z+ta,z+ta)} + \frac{|| \dbar_X \g ||^2}{ K_{t}(z+ta,z+ta)},
\end{equation}
where $\phi_a(t,z) = \phi (t, z+ta)$, $T_a$ is some positive $(n,n)$ form, 
and the $(n-1,0)$ form $\g$ is the solution of the equation $\d^{\phi}_X\g = \d^{\phi}_t k_t $ on each $D_t$.
Moreover, $\dbar_X\g$ is primitive w.r.t. the Euclidean metric. 

\end{theorem}

This formula (equation (\ref{intr-001})) can be compared with Berndtsson's previous calculation on the complex Hessian of the $Ding$-functional.
In particular, the two positive terms on the RHS of this equation is very similar to that case. 
However, this is indeed a pointwise formula. 
If we apply this formula to the sequence of weighted Bergman kernels $K^{(m)}_t$,
then the denominator $K^{(m)}(z,z)$ will converge to a delta function supported as $z$.  

The correct way is to do a blowing up analysis near the point $z$,
and then we can cook up a fiberwise holomorphic vector field by proving 
$$\dbar_X (\g^{(m)}/ K^{(m)}) \rightarrow 0$$
at this point. 
It gives a proof of the so called ``strict convexity" of the Mabuchi functional, provided enough regularities (Remark \ref{sc-rem-002}) on the geodesic.
This is exactly the question that we want to investigate in the second part of this paper.

\begin{conjecture}
Suppose the Mabuchi functional $\cM$ is linear along a $C^{1,1}$-geodesic segment $\cG$.
Then the geodesic $\cG$ is generated by a holomorphic vector field. 
\end{conjecture}

We call that the Mabuchi functional is strictly convex along the geodesic, if the conjecture is true.
It is well known that this strict convexity holds along a smooth geodesic.
However, there are two difficulties that prevent the generalisation
from the smooth geodesic case to the $C^{1,1}$-geodesic case: degeneracy and lack of regularities.  

Very recently, Berman \cite{Ber} constructed a counter-example of this conjecture when $\cG$
is essentially a $C^{1,1}$-geodesic ray on the two sphere $S^2$, 
and the geodesic ray in this example is ``very degenerate" in the sense that its restriction to each fibre $X_t \simeq S^2 , t>0$
vanishes on an open trip near the equator of $S^2$.

However, the situation on a geodesic segment may be very different from a geodesic ray.
For instance, every $C^{1,1}$ geodesic segment must be non-degenerate on a toric manifold.  

The \textbf{key observation} in this part is that the degeneracy can never happen 
if the following conditions hold on a $C^{1,1}$ geodesic $\cG$:
\begin{enumerate}
\item the Mabuchi functional $\cM$ is linear along $\cG$;

\item the two boundaries of $\cG$ are non-degenerate;

\item the volume form radio $f_{\vp_t}: = \omega_{\vp_t}^n / \omega^n$ of the geodesic is in $W^{1,2}(X_t)$ for each $t\in [0,1]$,
where $\omega_{\vp_t}: = \cG|_{X_t}$.
\end{enumerate}

A more precise statement can be found in Proposition  (\ref{prop-wg-002}).
In fact, the above conditions indeed hold if we further require the boundaries of $\cG$ are energy minimizers of the Mabuchi functional. 
Moreover, the condition of being non-degenerate energy minimizer also implies that the metric has to be a smooth $cscK$ on each fibre (\cite{CT}, \cite{HZ}).

However, the fiberwise smoothness of the geodesic is not enough to conclude the strict convexity,
and we also need to establish the regularities of the geodesic along time direction. 
Fortunately, a small perturbation result of all the $C^{5-\a}$ geodesics is established recently by Chen-Feldman-Hu \cite{CFH},
and the $cscK$ equation gives us a chance to utilise their result to establish regularities in time direction. 
Finally we proved the following 

\begin{theorem}
\label{thm-intr-002}
Suppose $\cG$ is a $C^{1,1}$ geodesic connecting two nondegenerate energy minimizers.  
Then the Mabuchi functional $\cM$ is strictly convex along $\cG$.
\end{theorem}

If we first apply the regularity theorem (\cite{CT}, \cite{HZ}) to the two energy minimizers on the boundary of $\cG$,
then this theorem reduces to prove the uniqueness of smooth cscK metrics (\cite{BB}, \cite{CPZ}).
However, unlike the previous ones, our new proof does not depend on the strict convexity of the so called $\mathcal{J}$-functional,
and we hope that this proof sheds some light to the ``intrinsic convexity" of the Mabuchi functional.

\textbf{Acknowledgment: }
The author is very grateful to Prof. M. P\u aun and Prof. B. Berndtsson who introduced this problem,
and he also wants to show his great thanks to Prof. J-P. Demailly, Prof. X.X. Chen and Prof. W. He for lots of useful discussion. 
Moreover, the author is supported by the ERC funding during writing this paper.

\section{The log-Bergman kernel}
Let $D$ be a pseudoconvex domain in $\bC_t \times \bC^n_z$, and let $\phi$ be a plurisubharmonic function in $D$.
For each $t$, the $n$-dimensional slice of $D$ is denoted by 
$$D_t: =\{z\in \bC^n |  \ (t,z)\in D \},$$ 
and the restriction of $\phi$ to each $D_t$ is 
$\phi_t(z): = \phi(t,z)$.
Consider the following Bergman space 
$$A_t: = A(D_t, e^{-\phi_t}), $$
consisting of all holomorphic functions in $D_t$, such that 
$$\int_{D_t} |h|^2 e^{-\phi_t} < \infty. $$
The Bergman kernel $K_t(\z, z)$ of $A_t$ for a point $z$ in $D_t$ is the unique holomorphic function of $\z$ satisfying 
$$ h(z) = \int_{D_t} h(\z) \ol{K_t(\z,z)} e^{-\phi_t(\z)} dV(\z),$$
for all function $h$ in $A_t$.

Berndtsson \cite{Bo} proved that the function $\log K_t(z,z)$ is plurisubharmonic, or identically equal to $-\infty$ in $D$. 
To see this, we will show a complete formula for the complex Hessian of this function on a product domain first.

\subsection{Product domain}
Let $V$ be a smoothly bounded strictly pseudoconvex domain in $\bC^n$, and $U$ be a domain in $\bC$.
Suppose $\phi$ is a smooth plurisubharmonic(Not necessarily being strictly plurisubharmonic!) function in a neighborhood of $V\times U$. 
Fix a point $z\in V$, and let $K_t(\cdot, z)$ be the Bergman kernel for $V$ with the weight function $\phi_t$.
Therefore, we have 
\begin{equation}
\label{bk-000}
h(z) = \int_{V} h(\z) \ol{K_t(\z,z)} e^{-\phi_t(\z)} dV(\z).
\end{equation}
Then the following observation is important 
\begin{equation}
\label{bk-001}
\Phi(t): = K_t(z,z) = \int_V K_t(\z,z)\ol{K_t(\z,z)} e^{-\phi_t}.
\end{equation}
By repeating Berndtsson's argument (Lemma 2.1, \cite{Bo}), it is easy to see 
that $K_t(\cdot, z)$ is also smooth as a function of $t$ for plurisubharmonic weight function $\phi$. 
Differentiate equation (\ref{bk-000}) in $t$ direction, and then the function 
$$  \d^{\phi}_t K_t: = e^{\phi} \frac{\d}{\d t} ( e^{-\phi} K_t ) $$
is orthogonal to each function $h$ in $A_t$. Furthermore, we define the following $(n,0)$ forms as 
$$k_t(\z): = K_t(\z,z) d\z^1\wedge\cdots\wedge d\z^n, $$
and 
$$ \d^{\phi}_t k_t: = \d^{\phi}_t K_t d\z^1\wedge\cdots\wedge d\z^n.$$

The $\dbar$ operator has closed range on a strictly pseudo-convex domain for any plurisubharmonic weight function.
(This is even true for any quasi-plurisubharmonic weight function whose curvature is bounded from below.)
Then there exists a unique $\dbar$-closed $(n,1)$ form $\a$ satisfying 
$$\d^{\phi}_t k_t = \dbar_X^*\a, $$ 
and $\a$ also smoothly depends on $t$ by a similar argument of $K_t$. 
Next, we can contract this $(n,1)$ form by the Euclidean metric, i.e. 
if $\a = \sum_{j=1}^n \a^j d\bar\z^j \wedge d\z$, then put $\g = \sum_{j=1}^n \a^j d\hat{\z}^j $.
Moreover, we have 
\begin{equation}
\label{bk-002}
\d^{\phi}_t k_t = \d^{\phi}_X \g, 
\end{equation}
and this is equivalent to the following equation 
\begin{equation}
\label{bk-003}
\d^{\phi} g =0, 
\end{equation}
where $g = k_t + dt\wedge\g$. Define a new $(n,n)$ form as 
$$T: = c_n g\wedge g e^{-\phi}. $$
Thanks to equation (\ref{bk-001}), the push forward of $T$ to $D_t = V$ equals to 
$$ [T]_{D_t}  = [c_n k_t\wedge \bar{k}_t]_{D_t} = K_t(z,z).$$

Berndtsson calculated the following formula 
\begin{equation}
\label{bk-004}
i\ddbar T = i\ddbar\phi\wedge T + c_n i \dbar g \wedge \ol{\dbar g} e^{-\phi}.
\end{equation}
Since $\dbar_X\g$ is primitive with respect to the Euclidean metric, we further have 
\begin{eqnarray}
\label{bk-005}
c_n i \dbar g \wedge \ol{\dbar g} &=& \left( \sum_{j,k\geq 1}^n \left| \frac{\d\a^j}{\d\bar\z^k} \right|^2 + \left| \frac{\d \a_0}{\d\bar \z_0}\right|^2 \right) idt\wedge d\bar t \wedge dV(\z)
\nonumber\\
&=&  ( |\dbar_X \g |^2 + | \dbar_t K_t |^2 ) i dt\wedge d\bar t\wedge dV(\z).
\end{eqnarray}
Therefore, we have the following after extending the current $T$ to $\bC^n$ (Lemma (4.1), \cite{Bo})
\begin{eqnarray}
\label{bk-006}
i\ddbar_t K_t(z,z) &\geq & [i\ddbar T]_{D_t}
\nonumber\\
&=& [i\ddbar \phi\wedge T]_{D_t} + \left( \int_{D_t} ( |\dbar_X \g |^2 + | \dbar_t K_t |^2 ) e^{-\phi_t} \right) idt\wedge d\bar t.
\end{eqnarray}

This is Berndtsson's calculation \cite{Bo} up to this stage, 
and then X. Wang \cite{Wang} generalised this formula to any pseudoconvex domain. 
However, we will proceed this computation as follows 
\begin{eqnarray}
\label{bk-007}
\ddbar_t\log K_t(z,z) &=& \Phi^{-2} \Big\{ \Phi\cdot \ddbar_t \Phi - \d_t\Phi\wedge \dbar_t \Phi\Big\}.
\end{eqnarray}
Then we claim the following 
$$\d_t\Phi (t) = \Big[ c_n g\wedge \ol{\dbar g} e^{-\phi} \Big]_{D_t}, $$
and the latter can be computed as 
\begin{eqnarray}
\label{bk-008}
&& \Big[ c_n (k_t + dt\wedge \g) \wedge ( dt\wedge \d_t \ol{k}_t   - d\bar t\wedge \d \bar\g) \Big]_{D_t}
\nonumber\\
&=& \Big[ k_t\wedge dt \wedge \ol{\dbar_{t}k_t} \Big]_{D_t} = (-1)^n \left( \int_{D_t} K_t \cdot \ol{\dbar_t K_t}e^{-\phi_t} \right) dt.
\end{eqnarray}
This is because none of the term in $\d_t\Phi(t)$ can contain $d \bar t$ direction. 
Moreover, we have 
\begin{equation}
\label{bk-009}
\d_t\Phi \wedge \ol{\d_t\Phi} = dt\wedge d\bar t \left|  \int_{D_t} K_t \cdot \ol{\dbar_t K_t} e^{-\phi_t} \right|^2
\end{equation}
Therefore, we have the identity  
\begin{equation}
\label{bk-010}
i\ddbar_t \log K_t(z,z) \geq \frac{[i\ddbar\phi\wedge T]_{D_t}}{[k_t\wedge\ol{k}_t]_{D_t}} + \frac{|| \dbar_X \g ||^2}{[k_t\wedge \bar{k}_t]_{D_t}} + i \k dt\wedge d\bar t,
\end{equation}
where $\k$ denotes by the following
\begin{equation}
\label{bk-011}
[k_t\wedge\bar{k}_t]^{-2}_{D_t}\left\{ \left(\int_{D_t} |K_t|^2 e^{-\phi_t} \right) \left( \int_{D_t} |\dbar_t K_t|^2 e^{-\phi_t} \right) - \left| \int_{D_t} K_t \cdot \ol{\dbar_t K_t} e^{-\phi_t}  \right|^2 \right\}.
\end{equation}
Thanks to the Cauchy-Schwarz Lemma, the term $\k$ is always positive. 
Therefore, we immediately see the following inequality
\begin{equation}
\label{bk-012}
i\ddbar_t\log K_t(z,z) \geq  \frac{|| \dbar_X \g ||^2}{K_t(z,z)},
\end{equation}
for all smooth weight function $\phi$ with $i\ddbar\phi \geq 0$. 

In fact, we can switch the derivatives and the push forward operator like $ \d_t \int T = \int \d T$,
only when the current $T$ is compactly supported in $\bC^n$. 
However, we have the following extension theorem of our currents by utilising the boundary condition.

Notice that it is true $\d\rho\wedge \g = 0 $ on the boundary, 
from solving the $\dbar$-Neumann boundary problem.
Then the claim is proved by a similar argument used in Lemma (4.1), \cite{Bo}. 
\begin{lemma}
\label{lem-bk-0001}
Let $\rho$ be a smooth real valued function in an open set $U$ in $\bC^n$.
Assume that $\d\rho\neq 0$ on $S: = \{z;\ \rho(z)=0 \} $ as a smooth real hypersurface.
Let $T$ be a real differential form of bidimension $(1,1)$ defined where $\rho<0$, 
with coefficients extending smoothly up to $S$. 

Assume that $\d\rho\wedge T$ vanishes on $S$,
and extend $T$ to a current $\tilde T$ in $\bC^n$ by putting $\tilde T =0$ where $\rho >0$. Then 
$$ i\d \tilde T = \chi_{\rho<0} i\d T. $$
\end{lemma}
\begin{proof}
The hypothesis on $T$ implies 
$$ \sum_j \rho_{,j} T_{j\bar k} = \rho c_k.$$

Let $w$ be a smooth function of compact support in $\bC^n$. Then we compute for each $k$
\begin{eqnarray}
\label{bk-0012}
\int_{\rho<0} i\d w \wedge d\bar z^k \wedge T &=& \int_{\rho<0} \sum w_{,j} T_{j\bar k}
\nonumber\\
&=& \int_{\rho=0} \sum_j w \rho_{,j} T_{j\bar k} dS/ |\d\rho| - \int_{\rho<0} \sum w \frac{\d T_{j\bar k}}{\d z^j},
\end{eqnarray}
but the boundary integration vanishes by our assumption, and we have 
\begin{equation}
\label{bk-00125}
\int_{\rho<0} i \d w \wedge d\bar z^k \wedge T = -\int_{\rho<0} w d\bar z^k \wedge \d T,
\end{equation}
and the result follows. 
\end{proof}

\subsection{Variation of the domain}
In the previous section, we proved that the $\log$-Bergman kernel function $\log K_t(z,z)$ 
is subharmonic in $t$ direction for any smooth plurisubharmonic weight function $\phi$ on a product domain. 
In this section, we will invoke Oka's trick of variation of the domain to compute the complex Hessian of $\log K_t(z,z)$ along other directions.
The idea runs as follows.

Let $F(t,z)$ be a function of two variables in $D$, and $a$ be a vector in $\bC^n_z$.  
Restrict this function on the one dimensional slice determined by the vector
$$v:= [1,a]$$
in $\bC_t \times \bC^n_z$, passing through the point $(z,t)$. That is to say, put a new function by translating the variables 
$$ F_a(t,z): = F(t, z+ta). $$

Then it is easy to see that the complex Hessian of $F$ in $v$ direction is equal to the complex Hessian of $F_a$ in $t$ direction, i.e. we have 
\begin{equation}
\label{bk-013}
\ddbar_t F_a = v^* \cdot (\ddbar F) \cdot v.
\end{equation}

Now we want to put this $F(t,z)$ to be our log-Bergman kernel function.
However, there is one important observation for Bergman kernels needed to be made. 

In fact, the translated Bergman kernel $K_t(z+ta, z+ta)$ is just the Bergman kernel at $z$ for the domain $D_t -ta$,
with the translated weight function. More precisely, translate the domain for each $t$-slice as $ D_{a,t}: = D_t - ta$,
and put $\phi_{a,t} (\z): =\phi_t(\z +ta) $. They naturally induce a new function space 
$$ A_{a,t}: = A (D_{a,t}, e^{-\phi_{a,t}}). $$
Moreover, there is a one-one correspondence between $A_{a,t}$ and $A_{t}$ as 
$$ h  \rightarrow h_{a}(\z): = h(\z + ta ),$$ since we have the following 
\begin{eqnarray}
\label{bk-014}
\int_{D_t} |h(\xi)|^2 e^{-\phi_t(\xi)} dV(\xi) &=& \int_{D_{a,t}} |h(\z+ ta)|^2 e^{-\phi_t(\z+ta)} dV(\z)
\nonumber\\
&=& \int_{D_{a,t}} |h_{a,t}(\z)|^2 e^{-\phi_{a,t}(\z)} dV(\z).
\end{eqnarray}

For later purpose, we begin with a \emph{parallelogram domain} first. 
Let $D$ be the following domain in $\bC_t\times \bC^n_z$ by defining each of its $t$-slice as 
\begin{equation}
\label{bk-0145}
D_t : = D_0 + ta,
\end{equation}
where $t$ is running over a small disk $U$ around the origin, and 
$D_0$ is a Euclidean ball $B_r(z)$ centered at $z$ with radius $r$. 
It is easy to see that the translated domain $D_{a,t}$ is precisely a product domain, i.e. 
$$D_{t,a} = D_0 $$ for each $t\in U$.
Then we denote $K_{a,t}(\z,z)$ by the Bergman kernel of the translated space $A_{a,t}$.
Moreover, we claim that the translated Bergman kernel is exactly equal to the Bergman kernel of the translated space as 
$$ K_t(\z + ta, z+ta) = K_{a,t}(\z,z). $$
This is because that we can compute the reproducing kernel as 
\begin{eqnarray}
\label{bk-015}
h_a(z) &=& \int_{D_{a,t}} h_a(\z) \ol{K_{a,t}(\z,z)} e^{-\phi_a (\z)} dV(\z)
\nonumber\\
&=& \int_{D_t} h(\xi) \ol{K_t (\xi, z+ta)} e^{-\phi_t(\xi)} dV(\xi)
\nonumber\\
&=& h(z+ta),
\end{eqnarray}
where $\xi: = \z +ta$. Then by the uniqueness of the Bergman kernel, we proved our claim. 

Therefore, the next goal is to compute $\ddbar_t \log K_{a,t}(z,z)$ on the product domain $U\times D_0 $.
In order to do this, we need to investigate the relation between the two $(n-1,0)$ form, $\g_{a}$ and $\g$ first.

Recall that the $(n-1,0)$ form $\g_a$ is uniquely determined by the equation
\begin{equation}
\label{bk-016}
\d^{\phi_a}_X \g_a = \d^{\phi_a}_t k_{a,t}, 
\end{equation}
and the condition that $\dbar_X \g_{a}$ is primitive with respect to the Euclidean metric.  
Put  
$$ \g_2(\z): = a \lrcorner k_{a,t} (\z )= \sum_{j=1}^n  K_t(\z + ta, z+ta) a^j d\hat{\z}^j.$$
Here $a$ is viewed as a vector field with constant coefficient on $D$,
and then we can prove the following 

\begin{lemma}
\label{bk-lem-001}
With above notations, we can decompose the $(n-1,0)$ from as $\g_a = \g_1 + \g_2$, such that $\g_2$ is $\dbar$-closed, 
and the pull back of $\g_1$ on $D_t$ satisfies 
\begin{equation}
\label{bk-017}
\d^{\phi}_X \g_1 (\xi) = \d^{\phi}_t k_{t}(\xi),
\end{equation}
where $\xi = \z +ta$.
Moreover, $\dbar_X \g_1$ is also primitive with respect to the Euclidean metric.
\end{lemma}
\begin{proof}
By the previous identification, we see that the $(n,0)$ form is equal to 
$$k_{a,t}(\z) = K_t(\z + ta, z+ta) d\z^1\wedge\cdots\wedge d\z^n,$$
and then we can compute as 
\begin{eqnarray}
\label{bk-018}
\d^{\phi_a}_t k_{a,t} &=& \d_t k_{a,t} - \d_t\phi_a\cdot k_{a,t} 
\nonumber\\
&=& \Big\{ (\d_t K_t)(\z+ ta,z+ta) + \d_X K_t \cdot a \Big\} d\z
\nonumber\\
&-&  \Big( \d_t\phi_t (\z+ta) + \d_X\phi_t \cdot a \Big) K_{t} d\z
\nonumber\\
&=& \Big\{ \d^{\phi}_t K_t(\z+ta, z+ta) + \d^{\phi}_X K_t(\z+ta, z+ta) \cdot a \Big\} d\z
\end{eqnarray}
This is because the Bergman kernel is anti-holomorphic on the second variable, i.e. $\d_{\ol{X}} K_t =0$. 
Then, it is easy to see from the definition of $\g_2$ that we have 
\begin{eqnarray}
\label{bk-019}
\d^{\phi_a}_X \g_2 &=& a \cdot \d_X K_{t} d\z - \d_X \phi(\z+ta)\wedge ( a\lrcorner k_{a,t} )
\nonumber\\
&=& a\cdot \d^{\phi}_X  K_t (\z +ta, z+ta) d\z.
\end{eqnarray}

Therefore, if we put $\g_1: = \g_a - \g_2$, then equation (\ref{bk-018}) and (\ref{bk-019}) gives the following 
\begin{equation}
\label{bk-020} 
\d^{\phi_a}_X \g_1 (\z) = \d^{\phi}_X \g_1 (\z +ta) =
 \d^{\phi}_t K_t(\z+ta, z+ta) \cdot d\z,
\end{equation}

More interestingly, observe that $\g_2$ is indeed holomorphic since we have 
\begin{equation}
\label{bk-021}
\dbar_X \g_2 = \sum_j \dbar_X K_t \wedge a^j d\hat\z^j = 0.
\end{equation}
Then we have $\dbar_X ( \g_1 + \g_2) = \dbar_X \g (\z+ta) $, and our result follows. 

\end{proof}

Based on this Lemma (\ref{bk-lem-001}), we will not distinguish $\g_a$ and $\g_1$ from now on. 
Combining it with equation (\ref{bk-010}), we proved the following inequality 
\begin{prop}
\label{bk-prop-001}
The log-Bergman kernel of a parallelogram domain $D$ has the following complex Hessian lower bound at one point 
\begin{equation}
\label{bk-022}
i\ddbar \log K_t (v,\bar v) \geq \frac{ [i\ddbar\phi_a \wedge T_a]_{D_{a,t}}}{ K_{a,t}(z,z)} + \frac{|| \dbar_X \g_1 ||^2}{ K_{a,t}(z,z)},
\end{equation}
for some positive $(n,n)$ form $T_a$, and $v$ is any tangent vector at this point. 
\end{prop}

\section{Smooth  geodesics}
The goal is to prove that a geodesic $\cG$ is generalized from a holomorphic vector field $V$ if the Mabuchi functional $\cM$ is linear along it. 
In fact, we have known what this vector field should be before hand. 
Suppose that the geodesic $\cG$ stands for the following metric on $X\times [0,1]\times S^1$
$$ \cG:= g_{t\bar t} dt\wedge d\bar t + g_{t\bar\b} dt\wedge d\bar z^{\b} + g_{\a\bar t}dz^{\a}\wedge d\bar t + g_{\a\bar\b} dz^{\a}\wedge d\bar z^{\b}. $$
Then it is a $S^1$-invariant solution of the so called \emph{Homogeneous complex Monge-Amp\`ere} (HCMA) equation,
 i.e. $\cG^{n+1} = 0$ on $X\times [0,1]\times S^1$. Moreover, this solution is uniquely determined if the boundary value is specified. 
In fact, we can also interpret the geodesic as such a solution on $X\times \Sigma$,
where $\Sigma$ is an annulus domain in $\bC$.

A local potential of the metric $\cG$ is a plurisubharmonic function $\phi$, i.e. we have 
$i\ddbar\phi = \cG$ in an open neighborhood of point $(z,t)\in X\times [0,1]\times S^1$.
Now the HCMA equation implies the following 
\begin{equation}
\label{sc-001}
c(\phi): = \ddot{\phi} - |\nabla \dot{\phi} |^2_g =0.
\end{equation}
Then our $S^1$-invariant vector field on $X\times [0,1] \times S^1$ can be written locally as 
$$ V_t := \frac{\d}{\d t} - \phi^{\bar\b\a}\dot{\phi}_{\bar\b} \frac{\d}{\d z^{\a}}. $$
In the following, the first goal is to prove that this vector field is holomorphic along each fiber $X\times \{ t \}$.

\subsection{The Mabuchi functional}
In this sub-section, we will assume the geodesic $\cG$ is smooth and non-degenerate along each fiber. 
The so called \emph{Mabuchi functional} or \emph{K-energy} is defined on the space $\cH$ as 
$$\cM (\vp): = \underline R \cE - \cE^{Ric\omega} + H, $$
where $\underline R$ is the average of the scalar curvature, and the first energy $\cE$ is 
$$ \cE (\vp): = \frac{1}{n+1} \sum_{i=0}^n \int_X \varphi \omega^i \wedge \omega_{\vp}^{n-i}.  $$
The second energy $\cE^{\a}$ is defined for any closed $(1,1)$ form $\a$ as
$$ \cE^{\a} (\vp): =  \sum_{i=0}^{n-1} \int_X \vp \omega^{i} \wedge \omega_{\vp}^{n-i-1}\wedge \a,$$
and the entropy part is defined as 
$$ H(\vp): = \int_X \left( \log\frac{\omega_{\vp}^n}{\omega^n} \right) \omega_{\vp}^n. $$

On a geodesic $\cG$, the complex Hessian of the Mabuchi functional is basically determined by the entropy part, 
and we usually write 
$$\cM(\varphi_t) \sim \int_X  \log\omega_{\varphi_t}^n\cdot\omega_{\varphi_t}^n,$$
where $\omega_{\varphi_t}: = \omega + i\ddbar_{X}\varphi_t$ is the restriction of the geodesic $\cG|_{X\times\{t\}}$ to the fiber.
Put $ \Psi_t: = \log \omega_{\varphi_t}^n$, and then it can be viewed as 
a metric on the canonical line bundle $K_X$ varying smoothly with respect to $t$. 

Therefore, its curvature $i\ddbar\Psi_t$ is a globally defined closed $(1,1)$ form,
and the complex Hessian of the Mabuchi functional along the geodesic $\cG$ 
can be computed as in \cite{BB}, \cite{CLP}
\begin{equation}
\label{sc-002}
i\ddbar_t \cM (\varphi_t) = [ i\ddbar\Psi\wedge \rho^n ]_{X_t}.
\end{equation}

In order to investigate the volume form $\omega^n_{\varphi_t}$, 
it is necessary to study the relation between the local Bergman kernels 
and the Monge-Amp\`ere mass of a plurisubharmonic function. 

Let $B_r(z_0)\subset \bC^n$ be the Euclidean ball centered at $z_0$ with radius $r$, 
and $\phi$ be a plurisubharmonic function in a neighborhood of $B_r(z_0)$. 
Let $K^{(m)}(\z,z)$ be the Bergman kernel of the space 
$$A_m: = A (B_r(z_0), \phi_m), $$ where $\phi_m: = m\phi$.
Then we call $K^{(m)}(\z,z)$ as the \emph{$m$-Bergman kernel} of the ball $B_r(z_0)$
with the weight function $\phi$.
Moreover, put $\b_m$ to be the radio of the $m$-Bergman kernel as 
$$\b_m: = \frac{c_n}{m^n} K^{(m)}(z,z) e^{-\phi_m(z)}. $$
Then it is well known that the following convergence is true
\begin{equation}
\label{sc-003}
\lim_{m\rightarrow \infty}\b_m (z) = \frac{(i\ddbar\phi)^n}{dV} (z).
\end{equation}
This convergence is smooth provided that $\phi$ is smooth up to the boundary.

\begin{rem}
\label{rem-sc-001}
Suppose the plurisubharmonic function $\phi$ has only $C^{1,1}$ regularities. 
Then Berman-Berndtsson \cite{BB} proved that 
the positive measure $d\mu_m: = \b_m dV $ is uniformly bounded from above.
More precisely, there exists a uniform constant $C(E,\phi)$ only depends on the domain $E$ 
and the upper bound of $\ddbar\phi$, such that it satisfies 
$$ d\mu_m \leq C(E,\phi) dV,$$
and the convergence (\ref{sc-003}) is in $L^1$ with respect to the Lesbegue measure. 
\end{rem}

Fix a point $p$ on the product manifold $X\times \Sigma$, 
and let $\Omega$ be a local coordinate system around $p$.
Assume $D\subset \Omega$ is a parallelogram domain defined as in equation (\ref{bk-0145}),
such that $p\in D_0$ is the center of the ball $B_r$.
On each slice $D_t$, we have the convergence as 
\begin{equation}
\label{sc-004}
\b_{m,t} (z) = \frac{c_n}{m^n} K^{(m)}_t(z,z) e^{-\phi_m} \rightarrow (i\ddbar_X \phi_t)^n(z).
\end{equation}
Moreover, this convergence is also smooth along $t$-direction. 
This is because we can write the $m$-Bergman kernel in its expansion form for a fixed point $z$ 
$$ K^{(m)}_t (z,z)  = a_0(t) m^n + a_1(t) m^{n-1} + \cdots + a_n(t) $$
and each coefficient $a_i$ smoothly depends on $t$ since the domain and the weight are both smooth in $(z,t)$.
Therefore, the following $(n+1,n+1)$ form 
\begin{equation}
\label{sc-005}
 \Big\{ i\ddbar \log K^{(m)}_t- i\ddbar\phi_m \Big\} \wedge (i\ddbar\phi)^n  = i\ddbar \log K^{(m)}_t \wedge (i\ddbar\phi)^n
\end{equation}
converges to the $(n+1,n+1)$ from
$$ T: = i\ddbar \log \omega_{\varphi_t}^n \wedge (i\ddbar\phi)^n $$
smoothly on $D$. 
Moreover, observe that the kernel of the $(1,1)$ form $\ddbar\phi$ is one dimensional,
and it is exactly generated by the vector field $v$. Then equation (\ref{sc-005}) can be re-written as 
\begin{equation}
\label{sc-006}
\ddbar \log K_t^{(m)} (v,\bar v) idt\wedge d\bar t \wedge (i\ddbar\phi)^n.
\end{equation}

Fix a point $p\in X\times \Sigma$, Proposition (\ref{bk-prop-001}) implies that the function 
$\ddbar \log K^{(m)}_t (v,\bar v) $ is non-negative for every $m$ at this point, 
weighted on any smooth function $\phi$ with $i\ddbar\phi \geq 0$. 
Therefore, the $(n+1,n+1)$ form $T$ is also positive at this point, by letting $m\rightarrow\infty$.


Suppose a non-degenerate geodesic $\cG$ has only $C^{1,1}$ regularities on $X\times \Sigma$. 
The positivity is proved in a similar way.
In fact, put the following positive currents as 
$$T_m: = i\ddbar \log \b_{m,t} \wedge (i\ddbar \phi)^n. $$
 
We claim that $T_m$ also converges to 
$T: = i\ddbar \log\omega_{\varphi_t}^n \wedge (i\ddbar\phi)^n$
in the current sense, by applying the Dominant Convergence Theorem(DCT). 

In order to prove the claim, we need to provide one subsequence of $\b_{m}$, 
such that it converges to the limit almost everywhere on $D$. 
But we have the following well known lemma 
\begin{lemma}
\label{lem-sc-add-001}
Let $f_m$ be a sequence of positive $L^{\infty}$ function on $D$.
Assume $\sup_D |f_m|$ is uniformly bounded,
and $f_m$ converges to some $f\in L^{\infty}(D)$ in distributional sense.
Then $f_m $ converges to $f$ almost everywhere on $D$ after passing to a subsequence.
\end{lemma}

Put $f_{m}: = \b_{k}$ and $f: = \frac{\omega^n_{\varphi_t}}{dV}$. The first condition in above Lemma is satisfied. 
Moreover, thanks to the Remark (\ref{rem-sc-001}), 
the function $f_{m,t}$ converges to $f_t$ in $L^1$ on each slice $D_t$.
Therefore, the second condition is also satisfied by Fubini's Theorem and DCT.

\subsection{Approximation vector fields}
Now assume the Mabuchi functional $\cM$ is linear along the geodesic.
Then the push-forward of the positive $(n+1,n+1)$ form $T$ to $\Sigma$ is zero. 
But this implies that the following non-negative function 
$$ \ddbar\log \omega_{\varphi_t}^n (v,\bar v) $$ 
is zero everywhere on each fiber $X_t$. 
(From this equation, one can prove $v$ is holomorphic by direct computation, but it needs at least $C^4$ regularities of the geodesic, \cite{CLP})

At a point $p\in X_t$, we can view the vector field $v_0$ as a vector in $\bC_t\times \bC^n_z$, i.e. put as before
$$v_0 = [1,a].$$
Then consider a small parallelogram domain $D$ centered at $p$ induced by the vector $a$ as defined in equation $(\ref{bk-0145})$.
Let $K^{(m)}_t$ be the $m$-Bergman kernel of the domain $D_t$ with the weight function $\phi$.
Then we have at this point $p$
\begin{equation}
\label{sc-008}
\ddbar \log K_t^{(m)}(v_0,\bar v_0) \rightarrow 0,
\end{equation}

Now Proposition (\ref{bk-prop-001}) immediately implies the following estimates 
\begin{equation}
\label{sc-009}
|| k^{(m)}_{a,t} ||^{-2}  [i\ddbar\phi_a \wedge T^{(m)}_a]_{D_{a,t}} = o (m^{-1}),
\end{equation}
and 
\begin{equation}
\label{sc-010}
|| k^{(m)}_{a,t} ||^{-2}  || \dbar_X \g^{(m)}_{a} ||^2 \rightarrow 0. 
\end{equation}

Observe that the function $K^{(m)}_{a,t}(\cdot, z)$ is holomorphic for a fixed $z$,
and then we can introduce the following probability measure for each $m$ (we will omit the index $a$)
$$ d\lambda_m (\z):= \frac{|K^{(m)}_t(\z,z)|^2 e^{-m\phi(\z)}}{K^{(m)}_t(z,z)} dV(\z).$$

Put
$$\hg^{(m)}: = \frac{\g^{(m)}}{K^{(m)}_t},$$
Equation (\ref{sc-010}) reduced to 
\begin{equation}
\label{sc-011}
\int_{D_t} |\dbar_X \hg^{(m)}|^2 d\lambda_m \rightarrow 0.
\end{equation}
Moreover, the $L^2$ norm of $\hg^{(m)}$ is uniformly bounded with respect to the measure $d\lambda_m$.
\begin{lemma}
\label{lem-sc-001}
There exists a uniform constant $C$, such that we have 
\begin{equation}
\label{sc-012}
\int_{D_t}| \hg^{(m)}|^2 e^{-m\phi} d\lambda_m \leq C.
\end{equation}
\end{lemma}
\begin{proof}
On each slice $D_t$ of the domain $D$, we consider the following function 
$$u_m:= \d^{\phi_m}_t K^{(m)}_t,$$
and let $f_m$ be its derivative in $\dbar$ direction, i.e. 
$$ f_m: = \dbar_X u_m = K^{(m)}_t \dbar_X \dot\phi_m. $$

Observe that $u_m$ is orthogonal to all holomorphic functions in the space $A_m$,
and then H\"ormander's $L^2$ estimate implies the following
\begin{eqnarray}
\label{sc-013}
|| u_m ||^2_{\phi_m} &\leq& m\int_{D_t} \phi^{\bar k j} \d_j\dot{\phi}_t \dbar_k\dot\phi_t |K^{(m)}_t|^2 e^{-\phi_m}
\nonumber\\
&\leq& m\int_{D_t} \ddot{\phi} |K^{(m)}_t|^2 e^{-\phi_m} 
\nonumber\\
&\leq& C m K^{(m)}_t(z,z),
\end{eqnarray}
where we used geodesic equation in the inequality. 
Moreover, recall that $\g^{(m)}$ is the solution of the following dual $\dbar_X$-equation
with the strictly positive weight function $m\phi$ 
$$ \d_X^{\phi_m} \g^{(m)} = \d^{\phi_m} k^{(m)}_t.$$
Therefore, H\"ormander's $L^2$ estimate again implies that we have 
\begin{equation}
\label{sc-014}
\int_{D_t} |\g^{(m)}|^2 e^{-\phi_m}\leq Cm^{-1} \int_{D_t} |\d^{\phi_m}_t K^{(m)}_t|^2 e^{-\phi_m}.
\end{equation}
Combining with equations (\ref{sc-013}) and (\ref{sc-014}), our result follows.
\end{proof}

\subsection{Blowing up analysis}
For each $m$, recall that the $(n-1,0)$ form is defined as the solution of the equation
\begin{equation}
\label{sc-015}
\d^{\phi_m}_X \g^{(m)} = \d^{\phi_m}_t k^{(m)}_t
\end{equation}
Take $\dbar_X$ on each side of the above equation, and we have 
\begin{equation}
\label{sc-016}
\ddbar \phi\wedge \g^{(m)} - \frac{1}{m} \d^{\phi_m}_X \dbar_X \g^{(m)} = \dbar_X\dot\phi_t \wedge k^{(m)}_t.
\end{equation}
Suppose $\g^{(m)}$ is holomorphic, and then the above equation immediately implies that 
$\hg^{(m)}$ is exactly induced by our target vector field $v$! 
This clue leads us to consider the limit of the sequence $\hg^{(m)}$.

Unfortunately, it seems that the probability measure $d\lambda_m$ converges to the Dirac mass $\delta$
at the origin (we take the fixed point $z$ to be the origin from now on). 
Therefore, there is less hope to find the lower bound of this sequence of measures. 
However, look at the density functions of these measures. Put
$$ \Phi_m(\z) := \frac{ |K^{(m)}_t (\z,0)|^2 e^{-m\phi} }{K^{(m)}(0,0)},$$
and at the origin, this density function behaves as 
\begin{equation}
\label{sc-017}
\Phi_m (0)= K^{(m)}_t(0,0) e^{-m\phi} \sim m^n 
\end{equation}

Now write the measure as 
$$ d\lambda_m(\z) =   \Phi_m(\z) d\z^1\wedge d\bar\z^1 \wedge \cdots \wedge d\z^n \wedge d\bar \z^n.  $$
Let us rescale the local coordinates by 
$$ \xi : = \sqrt{m} \cdot \z,$$
and the domain $B_r$ is rescaled to $B_{\sqrt m r}$ for the coordinate $\xi$. 
Then the measure in the new coordinate is 
$$ d\lambda_m (\xi) = \Psi_m (\xi)  d\xi^1\wedge d \bar\xi^1\wedge\cdots\wedge d\xi^n \wedge d\bar\xi^n,$$
where the new density function is 
$$ \Psi_m (\xi ): = \frac{1}{m^n} \Phi_m \left( \frac{\xi}{\sqrt{m}} \right).  $$
This is the correct rescaling, since we can see that the new density function at the origin is 
$$ \Psi_m(0) = \frac{1}{m^n} K^{(m)}_t(0,0) e^{-m\phi} \rightarrow (i\ddbar\phi)^n(0). $$ 

\begin{lemma}
\label{lem-sc-002}
With above notations, there exists a uniform constant $C$, such that we have on $B_R$
$$ | \log\Psi_m |_{C^{1}} \leq C.  $$ 
\end{lemma}
\begin{proof}
Notice that we have on $B_r$
\begin{equation}
\label{sc-018}
-mC \leq \ddbar_{\z} \log \Phi_m = -m \ddbar_{\z} \phi \leq 0, 
\end{equation}
and then on $B_{\sqrt m r}$, it implies 
\begin{equation}
\label{sc-019}
-C \leq \ddbar_{\xi}\log\Psi_m \leq 0.
\end{equation}

Put $u_m: = \log \Psi_m$, and then this sequence of function $u_m$ has uniformly bounded Laplacian on any 
Euclidean ball $B_R $, where the radius $R$ is less than $\sqrt m r$ for any $m$ large enough.
Moreover, we have its $L^p$ norm is controlled by 
\begin{eqnarray}
\label{sc-020}
\int_{B_R} |u_m|^p dV(\xi) &\leq & p! \int_{B_R} e^{u_m} dV(\xi) 
\nonumber\\
&\leq& C_p \int_{B_{\sqrt m r}} \Psi_m (\xi) dV(\xi)
\nonumber\\
&\leq& C_p.
\end{eqnarray}

Hence the $W^{2,p}$ norm of $u_m$ is uniformly controlled on the ball $B_R$.
Then by the Sobolev embedding theorem, the $C^1$ norm $|u_m|_{C^1(B_R)}$ is also uniformly bounded, 
by possibly passing to a subsequence. 
\end{proof}

Now we can assume the density function $\Psi_m$ is uniformly bounded away from zero on $B_R$.
Moreover, in the $\xi$-coordinate, we can further compute 
\begin{eqnarray}
\label{sc-021}
\d_{\xi} \log \Psi_m &=& \frac{ \ol{ K}^{(m)}\d (e^{-\phi_m } K^{(m)})    }{\Psi_m}
\nonumber\\
&=& \frac{\d_{\xi} K^{(m)}}{K^{(m)}} - \d\phi_m
\end{eqnarray}

Take the Taylor expansion of $\phi$ in a normal coordinate around the origin ($\phi\in C^3$ is enough)
$$ \phi(\z) = c_{j\bar k} \z^j\bar\z^k +O(|\z|^3).  $$
Then we denote $\psi_m$ by $m\phi$ in $\xi$-coordinate as 
$$\psi_m(\xi): = m\phi(m^{-1/2} \xi ) = c_{j\bar k}\xi^j\bar\xi^k +  O(m^{-1/2} |\xi|^3 ).$$
Therefore, we have easily 
\begin{equation}
\label{sc-022}
\d \psi_m (\xi) = O(|\xi |),
\end{equation}
and 
\begin{equation}
\label{sc-023}
\d_j \dbar_k \psi_m (\xi) =  c_{j\bar k} + O( m^{-1/2} |\xi |).
\end{equation}
Combining with equation (\ref{sc-021}) and (\ref{sc-022}), we have seen that there exists a uniform constant $C$ satisfying 
\begin{equation}
\label{sc-024}
\left| \frac{\d_{\xi} K^{(m)} }{K^{(m)}}\right|  \leq C,
\end{equation}
on $B_R$. 
Moreover, equation (\ref{sc-023}) implies that $\ddbar\psi_m$ converges to the $(1,1)$ form $\ddbar\phi(0)$ with constant coefficients on $B_R$, 
in $C^1$ norm provided $\phi\in C^3$.

Rescale the vector field $v$ by thinking of it as an $(n-1,0)$ form in the following way.
Writing $v|_{D_t} = \sum_{i} v^i d\hat\z^i$ on each fiber,
rescale gives for each $m$ 
$$ v_m (\xi): = m^{\frac{1}{2}}v^i (m^{-1/2}\xi ) d\hat\xi^i,$$
such that $v_m = m^{\frac{n}{2}} v$. Therefore, we have on the $\xi$-coordinate 
$$ \dbar_{X} v_m =  \sum_{i,k\geq 1}\d_{\bar k} v^i d\bar\xi^k\wedge d\hat\xi^i. $$

In the same way, we put $\tau_m: = m^{\frac{n}{2}} \hg^{(m)}$ on the $\xi$-coordinate.
Observe that $\dbar_X\tau_m$ is also primitive with respect to the Euclidean metric. 
Then its $L^2$ norm $|| \dbar_X \tau_m ||_{L^2(B_R)}$ converges to zero by equation (\ref{sc-011}) and Lemma (\ref{lem-sc-002}), 
since we have 
\begin{eqnarray}
\label{sc-025}
\int_{B_R} | \dbar_X \tau_m |^2  dV(\xi) &\leq& C \int_{B_R} |\dbar_X \tau_m|^2 \Psi_m dV(\xi)
\nonumber\\
&=& C \int_{B_R}  \left( \sum_{i,k} |\d_{\bar k} \tilde{a}^i (m^{-1/2}\xi)|^2 \right) d\lambda_m(\xi)
\nonumber\\
&\leq& C \int_{B_r} \left( \sum_{i,k} |\d_{\bar k} \tilde{a}^i (\z) |^2 \right) d\lambda_m(\z).
\end{eqnarray}

Moreover, equation (\ref{sc-016}) is rescaled into the following form
\begin{equation}
\label{sc-026}
\ddbar_X \psi_m\wedge \tau_m - \frac{1}{K^{(m)}}\d^{\psi_m}\dbar_X ( K^{(m)}\tau_m ) = \dbar_X \dot \psi_m \wedge d\xi.
\end{equation}
Since $\tau_m$ and $v_m$ are both $(n-1,0)$ forms, 
we can apply the contraction operator $i\ddbar_X \psi_m\lrcorner$ on both sides to get
\begin{equation}
\label{sc-add-0265}
\tau_m  - i\ddbar_X \psi_m\lrcorner   \left\{ \frac{1}{K^{(m)}}\d^{\psi_m}\dbar_X ( K^{(m)}\tau_m ) \right\} = v_m.
\end{equation} 
Hence it leads us to compare the current 
$$T^{(m)}_1: =  \dbar_X \tau_m, $$
with our target 
$$ T^{(m)}_2: =  \dbar_X v_m. $$

Observe that the current $T_1$ converges to the zero current in the $L^2$ sense on the ball $B_R$,
from equation (\ref{sc-023}) and estimate (\ref{sc-025}).
Moreover, the current $T_2$ converges weakly to the current $\dbar_X v(0)$ with constant coefficients. 
Then we close our proof from the following 
\begin{theorem}
\label{thm-sc-001}
Suppose $\cG$ be a non-degenerate smooth geodesic.
Assume that the Mabuchi functional is linear along it. 
Then the vector field $v$ is holomorphic on the geodesic. 
\end{theorem}
\begin{proof}
First we want to prove the vector is holomorphic along each fiber $X_t$.
Assume $p$ is the origin in the $\xi$-coordinate.  
It is enough to prove that for any smooth $(n-2,0)$ form $W$ compactly supported in $B_R$, we have as $m \rightarrow +\infty$
\begin{equation}
\label{sc-027}
\int_{B_R}  (\tau_m - v_m) \wedge \ol{\d^{\psi_m}_X W}e^{-\psi_m}  \rightarrow 0.
\end{equation}
But this is equivalent to prove the following converges to zero by equation (\ref{sc-026})
\begin{eqnarray}
\label{sc-028}
&&\int_{B_R}  \left\{  i\ddbar_X \psi_m \lrcorner \frac{1}{K^{(m)}} \d_X ( e^{-\psi_m} K^{(m)} \dbar_X \tau_m ) \right\} \wedge \ol{\d^{\psi_m}_X W}
\nonumber\\
&=& \int_{B_R} \{ i\ddbar\psi_m\lrcorner \d_X (e^{-\psi_m}\dbar_X\tau_m)  \}\wedge\ol{\d^{\psi_m}_X W} 
\nonumber\\
&+& \int_{B_R}  \left\{ i\ddbar\psi_m\lrcorner \left( \frac{\d K^{(m)}}{K^{(m)}} \wedge \dbar_X\tau_m\right) \right\} \wedge \ol{\d^{\psi_m}_X W} e^{-\psi_m}
\nonumber\\
&=& - \int_{B_R} \{ i\ddbar_X\psi_m\lrcorner  \dbar_X \tau_m\}  \wedge \ol{\dbar \d^{\psi_m}_X W} e^{-\psi_m} 
- \int_{B_R} \vartheta_m (e^{-\psi_m} \dbar_X\tau_m) \wedge  \ol{\d^{\psi_m}_X W} 
\nonumber\\
&+&\int_{B_R} \left\{ i\ddbar\psi_m\lrcorner \left( \frac{\d K^{(m)}}{K^{(m)}} \wedge \dbar_X\tau_m\right) \right\} \wedge \ol{\d^{\psi_m}_X W} e^{-\psi_m}.
\nonumber\\
\end{eqnarray}

Since we assumed the metric $\phi\in C^3$ is strictly plurisubharmonic, 
the K\"ahler metric $\omega_m: = i\ddbar\psi_m$ converges to the Hermitian form $c_{j\bar k} d\xi^{j}\wedge d\bar\xi^k$ in $C^1$-norm on $B_R$.
Therefore, the operator $i\ddbar\psi_m \lrcorner$ converges to the constant matrix $(c_{j\bar k})^{-1}$ multiplication,
and then the first and third term on the RHS of equation (\ref{sc-028}) converge to zero by the estimate (\ref{sc-025}) and (\ref{sc-024}) .

The second term is a bit of tricky. In order to deal with the operator $\vartheta_m$, 
we first write $*_m \tilde W: = \d^{\psi_m}_X W$, where $*_m$ is the Hodge star operator associated with the metric $\omega_m$.
Notice that the coefficients of the new form $\tilde W$ is just an linear combination of the coefficients of $\d^{\psi_m}W$ 
multiplying with the coefficients of $\omega_m$. 

Then an integration by parts says the following 
\begin{equation}
\label{sc-0285}
\langle \vartheta (e^{-\psi_m} \dbar_X\tau_m), \tilde W \rangle_m = \langle e^{-\psi_m} \dbar_X\tau_m, \dbar \tilde W \rangle_m,
\end{equation}
but the coefficients of the term $\dbar_X \tilde W$ involves at most one derivatives of the metric.
Therefore, all coefficients actually converges at least in $C^0$ to the limit,
and then it is easy to see equation $(\ref{sc-0285})$ also converges to zero.

Finally, we want to argue that the vector field $V_t$ is also holomorphic along the time direction, 
but this is straightforward for non-degenerate $\omega_{\phi}$, since we can view the derivative $\d V_t / \d \bar t$ 
as the complex gradient of $c(\phi)$ when $V_t$ is fiberwise holomorphic, i.e.
$$ \frac{\d V_t}{ \d \bar t} = \  \uparrow^{\omega_{\phi}} \dbar_X c(\phi). $$
Therefore, it vanishes when the function $c(\phi)$ is identically zero. 
 
\end{proof}

\begin{rem}
\label{sc-rem-002}
In fact, we can prove that the vector field is fiberwisely holomorphic 
with a weaker regularity assumption.
In fact, it is enough to assume the following: 
$\phi\in C^{1,1}$, $\phi \in C^{3}(X_t)$  and $\dot\phi \in W^{2,1}(X_t)$.
 
The regularity condition on $\dot\phi$ enable us to utilise the Lebegue Differentiation Theorem. 
Therefore, it is enough to prove $\dbar v =0$ in the weak sense as we did in equation $(\ref{sc-027})$.
Now the $\dbar$-operator still has the closed range locally along each fiber, 
but the weighted Bergman kernel $K_t(\z,z)$ is merely Lipschitz continuous in $t$-direction. 
Fortunately, the equation $\d^{\phi}_X\g = \d^{\phi}_t k_t$ only involves the first derivative of the weighted Bergman kernel,
and the solution $\g$ is in fact an $L^{\infty}$ function in $t$-direction. 
But this is already enough to establish equation $(\ref{sc-016})$, and everything follows similarly. 
 
\end{rem}


\section{Non-degeneracy of the $C^{1,1}$-geodesics}
The Mabuchi functional can be defined on the space of all $C^{1,1}$ K\"ahler potentials as 
$$ \cM: = E + H,$$
where the energy part is 
$$ E(\varphi): = \frac{R}{n+1}\cE(\varphi) - \cE^{Ric \omega} (\varphi),$$
and the entropy part is 
$$ H(\varphi): = \int_X \log \omega_{\varphi}^n \cdot \omega_{\varphi}^n - \int_X \log \omega^n \cdot \omega_{\varphi}^n. $$
They are both invariant under the normalization of the potential. 
Therefore, the Mabuchi functional is also normalization invariant. 

Let $\cG$ be a $C^{1,1}$ geodesic segment defined on the manifold $M: = X\times \Sigma$.
Our goal is to prove the following statement 
\begin{theorem}
\label{thm-wg-000}
Suppose $\cG$ is a $C^{1,1}$ geodesic connecting two non-degenerate energy minimizers of the Mabuchi functional.
Then $\cG$ is generated by a holomorphic vector filed.
\end{theorem}

This statement is not difficult to prove with enough regularities of the geodesic,
but the regularities of the geodesic can not be raised over $C^2$ in general.
In order to overcome this difficulty, we begin to investigate the truncated Mabuchi functional instead. 

If we restrict the Mabuchi functional to this geodesic,
then the principle contribution made to the 
complex Hessian of the Mabuchi functional is from the first part of the entropy
$$\cM(\varphi_t) \sim \int_X \log \omega_{\varphi_t}^n \cdot \omega_{\varphi_t}^n, $$
and we use $\Psi(\varphi): = \log \omega_{\varphi}^n$ to denote the metric on $K_{M/\Sigma}$. 
According to Berman-Berndtsson \cite{BB}, there also exists a fixed continuous metric $\chi$ on $K_{M/ \Sigma}$ satisfying 
\begin{equation}
\label{wg-001}
i\ddbar\chi \geq - m_0 (\pi^*\omega + i\ddbar \Phi),
\end{equation}
for some fixed integer $m_0$. 
Therefore, we can introduce the following truncated Mabuchi functional 
$$  \cM_A (\varphi_t) : = E (\varphi_t) + H_A(\varphi_t),$$
where 
$$H_A(\varphi): = \int_X \max \{ \log \omega_{\varphi}^n, \chi - A \}\omega_{\varphi}^n - \int_X \log \omega^n \cdot \omega_{\varphi}^n. $$

Put $\Psi_A: = \max \{ \log \omega_{\varphi}^n, \chi - A \}$, 
we observe that the complex Hessian of the truncated Mabuchi functional on the geodesic $\cG$ is the push forward 
$$  i\ddbar_t \cM_A : = [i\ddbar\Psi_A \wedge \rho^n]_{X_t}, $$
and Berman-Berndtsson \cite{BB} proved that the above $(n+1, n+1)$ current is positive. 
Hence the truncated Mabuchi functional is convex along the geodesic in distributional sense. 
Moreover, Theorem 3.4 in \cite{BB} also proved that this functional $\cM_A$ is continuous up to the boundary.
Therefore, the functional $\cM_A$ is indeed convex along the geodesic $\cG$.

\subsection{Gap phenomenon}
Put a function 
$f_{\varphi} : =  \omega_{\varphi}^n / \omega^n $, and the entropy can be re-written as 
$$ H(\varphi) = \int_X f_{\varphi} \log f_{\varphi} \cdot \omega^n. $$
Moreover, put 
$$f_A : = \max \left\{  \frac{\omega_{\varphi}^n}{\omega^n}, \frac{e^{\chi -A}}{\omega^n}  \right\}, $$
and then the truncated entropy is 
$$ H_{A}(\varphi) = \int_X f_{\varphi}\log f_A\cdot \omega^n. $$
Therefore, by dominant convergence theorem, it is easy to see that the truncated Mabuchi functional 
is decreasing to the Mabuchi functional on a geodesic $\cG$, 
i.e. $\cM_A(\varphi_t) \searrow \cM(\varphi_t)$ for each $t\in[0,1]$ as $A\rightarrow +\infty$.

\begin{lemma}
\label{lem-wg-001}
Suppose the Mabuchi functional $\cM$ is linear along a geodesic segment $\cG$. 
Then there exists a number $A_0$, such that for each $A \geq A_0$, 
the truncated Mabuchi functional $\cM_A$ coincides with $\cM$ along the geodesic $\cG$.   
\end{lemma}
\begin{proof}
Without loss of generality, we can assume that the Mabuchi functional is identically zero along the geodesic, i.e. $\cM(\varphi_t) \equiv 0$ for each $t\in [0,1]$. 

Let $\varphi_0, \varphi_1$ be the boundary value of the potentials of the geodesic segment $\cG$,
and then we can pick up a constant as 
$$A_0 :=  \sup_{X} (\chi - \log f_{\varphi_i} ), $$
for $i = 0,1$.
Observe that on the boundary, we have 
\begin{equation}
\label{wg-002}
\cM_A (0) = \cM (1),\ \ \  \cM_A(1) = \cM(1),
\end{equation}
for all constant $A \geq A_0$. 
Therefore, $\cM_A(0) = \cM_A(1) =0$.

Thanks to the convexity and continuity of the truncated Mabuchi functional, 
we conclude that $\cM_A (t)\leq 0$ for each $t\in [0,1]$.
However, we also know that $\cM_A\geq \cM$ from the definition of the truncated Mabuchi functional.
Therefore, the truncated Mabuchi functional is also identically zero along the geodesic $\cG$, i.e. 
$\cM_A(\varphi_t) \equiv 0,  $ for each $t\in [0,1]$.
\end{proof}

An immediate consequence of Lemma {\ref{lem-wg-001}} is that the truncated entropy $H_A(\varphi_t)$
also coincides with the entropy $H(\varphi_t)$ along the geodesic $\cG$ for each $A\geq A_0$.  
Therefore, it gives a way to describe the degenerate locus of the geodesic along each fiber $X_t$.

\begin{prop}
\label{prop-wg-001}
Suppose the Mabuchi functional is linear along a geodesic $\cG$.
Then on each fiber $X_t$, there exists a non-empty subset $\cZ_t$ and a uniform constant $A_0$,
such that the following holds up to a set of measure zero:
$$ \omega_{\varphi_t}^n \geq e^{\chi -A_0} \omega^n,$$
on $\cZ_t$, and $\omega_{\varphi_t}^n = 0$ on $X_t - \cZ_t$.
\end{prop}
\begin{proof}
Denote the set $\cZ_t$ by the non-zero locus of the function $f_{\varphi}$ on the fiber $X_t$.
It is a non-empty set since the integral of $\omega_{\varphi_t}^n$ is the fixed volume of this K\"ahler metric.
Consider the level sets of the function $f_{\varphi}$ on a fiber $X_t$ as 
$$ E_A: =\{ p \in X_t ;\  f_{\varphi} (p) \geq e^{\chi(p) -A} / \omega^n \}. $$

We claim that $E_A - E_{A_0}$ has measure zero for each $A > A_{0}$.
Then we have up to a set of measure zero 
$$ \cZ_t = \bigcup_{A \geq A_{0}} E_A = E_{A_0},$$
and our result follows.

Now we have on each fiber $X_t$ 
\begin{eqnarray}
\label{wg-002}
0= H_{A_0} (\varphi_t) - H(\varphi_t) &=& \int_X (\log f_{A_0} - \log f) f 
\nonumber\\
&\geq&  \int_{E_A } (\log f_{A_0} - \log f) f 
\nonumber\\
&\geq& \ep  \int_{E_A } (\log f_{A_0} - \log f),
\end{eqnarray}
but this inequality implies that $f_{A_0} = f$ almost everywhere on $E_A$
since $\log f_A \geq \log f$. Therefore, the set $E_{A} - E_{A_0}$ indeed has measure zero,
and we complete the claim. 

\end{proof}

In other words, the value of the volume form radio $f_{\vp}$ on the geodesic  
is either larger than a positive constant $\ep_0$ or equal to zero almost everywhere,
provided the Mabuchi functional is linear. 
That is to say, there exists a \emph{gap} for the value of the volume form radio on $X\times \Sigma$. 

\subsection{$W^{1,2}$ estimate}
The next step is to investigate the regularities of the volume form radio $f_{\varphi} = \omega^n_{\vp}/\omega^n$. 
Here we will invoke the following Theorem from Chen-Tian \cite{CT} to estimate the $W^{1,2}$ norm of $f_{\vp}$. 
However, it requires that the Mabuchi functional realizes its minimum, if the minimum exists. 

\begin{prop}[Chen-Tian]
\label{thm-wg-001}
Suppose $\varphi\in PSH(X, \omega)$ has $C^{1,1}$ regularities,
and the Mabuchi functional has a uniform lower bound in this K\"ahler class.
Meanwhile, if the energy $\cM(\varphi)$ realizes this minimum, 
then the square root of the volume form radio $ f_{\vp}^{1/2} $ is in $W^{1,2}(X,\omega)$. 
\end{prop}

This regularity theorem is proved via the so called ``\emph {weak K\"ahler Ricci flow }",
and a crucial step is to observe that the first derivative of the energy $ \frac{\d\cM }{\d t} |_{t=0}$
is bounded from below uniformly along the flow direction. 
However, this is unlikely to be true if we merely assume the Mabuchi functional is linear. 

Combining with the ``gap phenomenon" proved in last section, this $W^{1,2}$ estimate enable us to conclude
the following uniform non-degeneracy of the geodesic, provided the two boundary datas are non-degenerate. 

\begin{prop}
\label{prop-wg-002}
Suppose the following conditions hold:
\begin{itemize}
\item
the Mabuchi functional $\cM$ is linear along a $C^{1,1}$ geodesic $\cG$;
\item
the two boundaries of $\cG$ are non-degenerate;
\item
for some $\a >0$, assume that
$f_{\vp_t}^{\a}$ is in  $W^{1,2}(X_t)$ for each $t\in [0,1]$.
\end{itemize}

Then there exists a uniform constant $\ep_0$, such that we have
$$  f_{\vp_t} > \ep_0,$$
everywhere for each $X_t$.
\end{prop}
\begin{proof}
For any fiber $X_t$, put $u =  f_{\vp}^{\a}$. 
Thanks to Proposition (\ref{prop-wg-001}), there exists a uniform constant $\ep_0$,
such that the gap phenomenon occurs for the value of the function $u$ almost everywhere on $X_t$. 
Since the function $u$ itself is in $L^{\infty}$, we can further assume that 
the gap exists for each point $p\in X_t$, i.e. either
$u(p) > \ep_0$ or $u(p)=0$. 
Moreover, the function $u$ is also non-trivial since $\int_{X_t} u^2 =1$. 

Let $U_i \subset V_i \subset U'_i$ be open coverings of $X_t$, 
such that $U_i$ corresponds to a ball $B_1$ with radius $1/2$, $U'_i$ corresponds to a ball $B_2$ with radius $2$,
and $V_i$ corresponds to the $n$-interval $[0,1]^n$ in each local coordinate.  
Then a standard regularity argument (see Lemma (\ref{lem-app2-001})) implies that each restriction $u_i: = u|_{V_i}$
is either trivial or greater than $\ep_0$ everywhere on $V_i$. 
Since $U_i$ forms an open converging of $X_t$, only the latter situation can occurs for each $u_i$.  
Then our result follows.

\end{proof}

As a simple corollary of Propositions (\ref{prop-wg-001}) and (\ref{prop-wg-002}),
we proved the following.

\begin{corollary}
Any $C^{1,1}$ geodesic $\cG$ connecting two non-degenerate energy minimizers of the Mabuchi functional 
must be uniformly non-degenerate.
\end{corollary}

\section{Energy minimizers and smoothness}

Due to the work of Chen-Tian (Section 8, \cite{CT})
or the recent work of He-Zeng \cite{HZ},
a non-degenerate $C^{1,1}$ energy minimizer of the Mabuchi functional 
is indeed smooth and satisfies the $cscK$ equation.
For the convenience of the readers, we recall an argument here. 
The first step is to deduce $cscK $ equation in the weak sense for the energy minimizers.

\begin{lemma}
Suppose $\omega_{\vp}$ is a non-degenerate $C^{1,1}$-energy minimizer of the Mabuchi functional. 
Then for any locally supported test function $\chi$, it satisfies 
\begin{equation}
\label{wg-003}
\int_X \log f_{\vp} i\ddbar \chi\wedge \omega_{\vp}^{n-1} = \int_X \chi (Ric (\omega ) - \omega_{\vp}) \wedge \omega_{\vp}^{n-1}.
\end{equation} 
\end{lemma}
\begin{proof}
It is easy to see that the potential $\vp$ is a strictly $\omega$-$psh$ function on $X$,
and we actually can gain a bit more than this from the non-degeneracy. 
The non-degenerate condition on the volume form $f_{\vp} > \ep$ and the upper bound of the coefficients of the metric $\omega_{\vp}$ together implies the lower bound of the metric, i.e. $\omega_{\vp} > \ep' \omega$ for some small $\ep' >0$.
Therefore, the potential $\vp$ is actually $(1- \ep' )\omega$-$psh$ on $X$.

According to Demailly's regluarization theorem \cite{Dem},
there exists a sequence $\vp(s)\in PSH^{\infty}(X, (1-\ep')\omega)$, such that
$\vp(0) = \vp$, $|\vp(s)|_{C^{1,1}} < C$, and $\vp(s)$ converges to $\vp$ in $W^{2,p}$ for any $p$ large.

For any small $s>0$, we want to construct a smooth curve $\vp(s, t)\in PSH^{\infty}(X,\omega)$ initiated from $\vp(s)$,
such that it satisfies 
\begin{equation}
\label{wg-004}
\left. \frac{\d \vp(s,t)}{\d t}  \right|_{t=0} = \chi,
\end{equation}
and it exists for $t\in[0,\delta)$ for some uniform small constant $\delta$.
In fact, one can check that a linear combination $\vp(s,t): =\vp(s) + t \chi $ would work, since we have
$$\omega + i\ddbar\vp(s,t) > \ep'\omega+ t i\ddbar\chi \geq 0, $$
for all $t$ small enough. 

Since the potential $\vp$ is an energy minimizer, we have 
$$ \cM(\vp) =  \lim_{s\rightarrow 0} \cM(\vp(s,0)) \leq \cM(s,t), $$
for any $s>0, t\geq 0$. Therefore, for any small constant $c >0$, 
there exists a sequence of points $s_i, t_i \rightarrow 0$ such that we have 
\begin{eqnarray}
\label{wg-005}
-c &\leq& \left. \frac{\d \cM(s,t)}{\d t} \right|_{s_i,t_i}
\nonumber\\
&=& \int_X \log f_{\vp} ( \Delta_{\vp} \chi ) \omega_{\vp}^n |_{s_i,t_i} - \int_X \chi (Ric(\omega)- \omega_{\vp})\wedge \omega_{\vp}^{n-1}|_{s_i,t_i}.
\end{eqnarray}
Since each coefficient of $\omega_{\vp(s_i,t_i)}$ converges strongly to $\omega_{\vp}$ in $L^p$,
we conclude the following one side inequality by letting $s_i, t_i \rightarrow 0$
\begin{equation}
\label{wg-006}
\int_X \log f_{\vp} i\ddbar \chi\wedge \omega_{\vp}^{n-1} \leq \int_X \chi (Ric (\omega ) - \omega_{\vp}) \wedge \omega_{\vp}^{n-1}.
\end{equation}
Observe that both sides of our equation (\ref{wg-005}) are linear in $\chi$,
and then we obtain an inequality with reversed sign of (\ref{wg-006}) by putting $\tilde\chi = - \chi$.
Hence the desired equation follows.

\end{proof}

Since the volume form radio $f_{\vp}$ is in $W^{1,2}$,
we can rewrite our equation (\ref{wg-005}) locally in the weak sense as 
\begin{equation}
\label{wg-007}
\frac{1}{f_{\vp}}\frac{\d}{\d\bar z^{\b}} \left( g^{\bar\b\a}_{\vp} f_{\vp} \frac{\d}{\d  z^{\a}} \log f \right) = 
g_{\vp}^{\bar\b\a} R_{\a\bar\b} -n.
\end{equation}
This is a second order elliptic differential equation, 
and then the smoothness of the potential $\vp$ is obtained by boot strapping the regularities as in \cite{CT}. 
Finally we have proved the following 

\begin{theorem}
\label{wg-thm-001}
Suppose $\cG$ is a $C^{1,1}$ geodesic connecting two non-degenerate energy minimizers of the Mabuchi functional.
Then the restriction of the geodesic $\omega_{\vp_t}$ to each fiber $X_t$ is a smooth $cscK$ metric for all $t\in [0,1]$.
\end{theorem}

\subsection{Time direction}
We are going to understand the regularities in time direction of the geodesic connecting two non-degenerate $C^{1,1}$ energy minimizers.
It is a well known question that whether we can perturb a $C^{1,1}$ geodesic a little to get a smooth one.
Recently Chen--Feldman-Hu \cite{CFH} proved the following theorem, and partially answered this question in a local version.

\begin{theorem}[Chen-Feldman-Hu]
\label{wg-thm-003}
Let $\vp$ be a smooth K\"ahler potential in the space $\cH$.
For any real number $\a\in (0,1)$, there exists a small number $\ep>0$, such that for any K\"ahler potential $\vp_1 \in C^5$ 
satisfying $|\vp_1 - \vp|_{C^{5}} < \ep$, the geodesic $\cG$ connecting $\vp$ and $\vp_1$ is non-degenerate and has $C^{5-\a}$
regularities in both time and space direction.
\end{theorem}

In our case, the time direction regularities of the geodesic 
will be improved to $C^{5-\a}$ by invoking this theorem, provided that we can prove that the geodesic potentials 
$\vp_t$ and $\vp_{t'}$ is close in $C^5$ norm if $t$ and $t'$ is close enough.
However, this statement is not clear in general, even if we assume $\vp_t$ is smooth along the fibre $X_t$ for every $t$,
and this is where the $cscK$ equation comes into play. 

\begin{prop}
\label{prop-wg-005}
Suppose $\cG$ is a $C^{1,1}$ geodesic connecting two non-degenerate energy minimizers of the Mabuchi functional.
Then the geodesic is at least $C^{5-\a'}$ continuous in both space and time direction.  
\end{prop}
\begin{proof}
Locally the $cscK$ equation can be decomposed into two couple second order equations as 
\begin{equation}
\label{wg-008}
\det (g_{\a\bar\b} + \d_{\a}\d_{\bar\b}\vp_t) = e^{F_{\vp_t}},
\end{equation}
and 
\begin{equation}
\label{wg-009}
\Delta_{\vp_t} F_{\vp_t} = \underline{R}.
\end{equation}
Thanks to Proposition (\ref{prop-wg-002}), the function $| F_{\vp_t} |$ and $|\Delta \vp_t|$ are uniformly bounded
along the $C^{1,1}$ geodesic.
Moreover, these conditions also imply that 
the metric is uniformly equivalent along the geodesic, i.e. there exists a constant $C>0$ satisfying 
$$ C^{-1}\omega < \omega_{\vp_t} < C \omega, $$ for any $t\in [0,1]$,
and then $| \Delta F_{\vp_t}| $ is also uniformly bounded by equation (\ref{wg-009}).
Then a standard elliptic estimate implies that the potential is uniformly bounded in $W^{4,p}$.
In particular, the norm $|\vp_t|_{C^{3,\a}}$ is uniformly bounded by some $1> \a >0$ by the Sobolev embedding theorem.

Notice that the function  $F_{\vp}$ is exactly equal to $f_{\vp} - \log \det g$ locally, and then
equations (\ref{wg-008}) and (\ref{wg-009}) can also be written as 
\begin{equation}
\label{wg-010}
g_{\vp_t}^{\bar\b\a} \frac{\d^2}{\d z^{\a} \d \bar z^{\b}} \log f_{\vp_t} = \underline{R} - tr_{\omega_{\vp_t}} Ric (\omega).
\end{equation}
It follows from the Schauder estimates that $f_{\vp_t}$ is uniformly bounded in $C^{1,\a}$,
and then $| \vp_{t}|_{C^{5,\a}}$ is also uniformly bounded.

Now fix a time $t_0\in [0,1]$, and consider a sequence of points $t_i$ such that $t_i\rightarrow t_0$ as $i\rightarrow \infty$.
Then the potential $\vp_{t_i}$ converges to a $C^{5,\a}$ potential $\vp_{\infty}$ in $C^5$ norm by possibly passing to a subsequence.
However, since $\vp_t$ is continuous in time direction along the geodesic, the limit $\vp_{\infty}$ must coincide with the potential $\vp_{t_0}$.
Therefore, for any $\ep>0$, we have $ | \vp_{t_i} - \vp_{t_0}|_{C^5} < \ep$ for all $i$ large. 

\end{proof}

Finally, a direction computation on the complex Hessian of the Mabuchi functional 
proves our Theorem (\ref{thm-wg-000}), once provided that the geodesic $\cG$ has more than $C^4$ regularities.

\subsection{yet another proof}

There is another way to figure out the regularities of the geodesic along the time direction,
by using the $cscK$ equation. Write the equation along the geodesic as follows 
\begin{equation}
\label{wg-011}
g^{\bar\b\a}_{t} \d_\a \d_{\bar\b} \log\det g_t = \underline R. 
\end{equation}
Then take the first time variation of the family of the $cscK$ equations yielding as 
\begin{equation}
\label{wg-012}
 \Delta^2_{\phi} (\delta_t\phi ) - R^{\bar\b\a} (\delta_t\phi)_{,\a\bar\b} = 0.
\end{equation}
This is a one parameter family of the fourth order (strict) elliptic equations, with uniformly bounded coefficients. 
Hence we can lift the regularities of the function $\delta_t\phi$ in the space direction by the standard elliptic estimates.
However, this equation (\ref{wg-012}) can not be derived in the usual sense, since $\delta_t\phi$ is merely a Lipschitz continuous function on $X\times \Sigma$.

What we can do is to take the difference quotient along the time direction in the $cscK$ equation (\ref{wg-011}), i.e. 
$$ \delta_t \phi :=  \frac{\phi(t_0+ t, \cdot) - \phi(t_0, \cdot)}{t}. $$
The difference quotient also satisfies the Leibniz rule, and then equation (\ref{wg-012}) indeed holds for it. 
Finally, observe that all the elliptic estimates coming from equation (\ref{wg-012}) holds uniformly for $t$.
Therefore, we actually improved the regularities of the function $\dot\phi$ in the space, and proved $\dot\phi$ is in fact a smooth function on $X_t$.

We circumvent using the Chen-Feldman-Hu Theorem, 
but the $cscK$ equation plays an essential role in this argument.
Suppose the two boundary values of $\cG$ are no longer energy minimizers. 
Then we lost the $cscK$ equation in general, and it still remains a question to prove the holomorphicity of the vector field $V_t$.

\section{Remarks}
To prove the fiberwise holomorphicity of the vector field,
the full regularities of the geodesic is in fact not needed, 
as we can see from the argument of Theorem (\ref{thm-sc-001}) and Remark (\ref{sc-rem-002}).
In fact, it is hopeful to further weaken this condition to $C^{1,1}$ regularities. 
However, the non-degeneracy of the geodesic indeed played a crucial role in our proof. 
Therefore, the first question we would like to ask is whether the geodesic segment becomes non-degenerate,
when that the Mabuchi functional is linear along it. 

This is not true for a general geodesic ray, as Berman's example (\cite{Ber}) have shown. 
Technically speaking, this is exactly because the $W^{1,2}$ estimate (the last condition in Proposition (\ref{prop-wg-002})) fails 
in this example. 
For a suitable approximation $\vp_{\ep}$ of the geodesic potential $\vp$,
it is nature to consider the following energy around a point $p\in X$
$$ G_{\ep, r}(p) := \fint_{B_r(p)} |\nabla f_{\vp_{\ep}}|^2, $$
and 
$$ G_{\ep} (p): = \lim_{r\rightarrow 0} G_{\ep, r}(p).$$
When the volume form radio jumps, it is exactly the place where this energy blows up to infinite. 
In other words, suppose $\cZ: = \{ f_{\vp}>0 \}$ is the non-degenerate locus of the volume form radio,
and then we expect to have 
$$ \d\cZ = \{ p\in X;\ \ \lim_{\ep\rightarrow 0}G_{\ep}(p) = \infty \}.   $$

Therefore, it might be helpful to study this energy for the purpose of investigating the behaviour of the degenerate locus,
when the Mabuchi function is linear. For example, if we can prove the above set is closed, then 
it is very likely that the degenerate locus $\cZ^{c}$ is an open set of the manifold, and this is completely unknown before.

\section{Appendix }
Let $\Omega$ denote the domain $n$th. unit interval $(0,1)^n$ in $\bR^n$, 
and $u$ is a non-negative function on $\Omega$. 
We say the function $u$ is trivial if it is identically zero outside a set of measure zero.

\begin{lemma}
\label{lem-app2-001}
Suppose a non-trivial function $u$ belongs to the intersection of the spaces 
$W^{1,2}(\Omega)$ and $L^{\infty}(\Omega)$, and satisfies the following gap condition: 
$$ u\geq 1,$$
whenever $u \neq 0$. Then $u \geq 1$ everywhere on $\Omega$. 
\end{lemma}
\begin{proof}
First we can assume the function $u$ is either equal to $1$ or $0$,
otherwise replace $u$ by $(1-u)_{+}$, and the energy $\int_{\Omega} |\nabla u|^2$ decreases 
under this change by the property of maximum operator \cite{GT}.

Assume $n=2$, and $\Omega$ is the unit square $(0,1) \times (0,1)$. We will first illustrate our idea in this case. 
The energy can be decomposed as follows 
\begin{equation}
\label{app2-001}
\int_{\Omega} |\nabla u|^2 dxdy = \int_{\Omega} |D_{x} u|^2 + \int_{\Omega} |D_{y} u|^2.
\end{equation}
Hence we have 
$$ \int_0^1 \left( \int_0^1 |D_{x} u|^2 dx \right) dy <  +\infty, $$
and then Fubini's Theorem implies that for almost everywhere $y\in (0,1)$, we have 
$$ \int_{0}^1 |D_x u|^2(x,y) dx < +\infty.$$
Therefore, the restriction of $u$ on the interval $(0,1) \times \{ y\}$ is $W^{1,2}$, 
and then $u$ is continuous along this interval, thanks to the Sobolev embedding theorem. 
That is to say, for almost everywhere $y\in (0,1)$, the function $u(x,y)$ is identically $1$ or $0$ on the slice $(0,1) \times \{ y\}$.

On the other hand, we have from equation (\ref{app2-001}) 
$$ \int_0^1 \left( \int_0^1 |D_{y} u|^2 dy \right) dx <  +\infty.$$

As we have proved, for any $x_1, x_2\in (0,1)$, the function $u(x_1,y)$ equals to $u(x_2, y)$ for almost everywhere $y \in (0,1)$.
Therefore, the partial derivatives satisfy 
$$ \int_0^1 | D_y u |^2 (x_1, y) dy  = \int_0^1 | D_y u|^2 (x_2, y)dy,$$
and then we have for each $x_0\in (0,1)$
\begin{equation}
\label{app2-002}
\int_0^1 |D_y u |^2 (x_0, y) dy < +\infty,
\end{equation}
The same argument implies that the restriction $u_{\{x_0 \} \times (0,1)}$ is also continuous,
and therefore $u$ is identically equal to $1$ on the square since it is non-trivial. 

For $n>2$, we can also decompose the energy as 
$$ \int_{\Omega} |\nabla u|^2 dx_1\cdots dx_n = \int_0^1 \int_{\Omega^{n-1}} |\nabla_{n-1} u|^2  + \int_{\Omega^{n-1}}\int_0^1 |D_{x_n} u|^2, $$
and then the result follows in a similar way by induction on the dimension of $\Omega$. 
\end{proof}

\end{document}